\documentclass[a4paper,10pt]{article}

\usepackage[utf8]{inputenc}
\usepackage[english]{babel}

\usepackage{mathtools}
\usepackage{amsmath}
\usepackage{amssymb}
\usepackage{dsfont}
\usepackage{amsthm}
\usepackage{cases}
\usepackage{stix}
\usepackage{pgf,tikz}
\usetikzlibrary{arrows}\usepackage[pdftex,colorlinks]{hyperref}
\hypersetup{                    
           citecolor=red,     
             filecolor=blue,     
             linkcolor=blue,     
             urlcolor=blue}

\usepackage[textwidth=6in,textheight=9.25in, paperwidth=7.5in,paperheight=11.25in, inner=2.5cm, outer=2.5cm]{geometry} % add in the following: inner=2.0cm, outer=1.5cm
\usepackage{amsfonts}
\usepackage{enumitem}
\usepackage{amsthm}
\usepackage{amsmath}
\usepackage{amssymb}

\newtheorem{lemma}{Lemma}[section]
\newtheorem{proposition}[lemma]{Proposition}
\newtheorem{theorem}[lemma]{Theorem}

\newtheorem{principle}{Principle}[section]

\theoremstyle{definition}

\theoremstyle{definition} \newtheorem{defn}{Definition}[section]
\theoremstyle{remark} 
\theoremstyle{plain} \newtheorem{prop}[defn]{Proposition}
\theoremstyle{plain} 
\theoremstyle{plain} 
\theoremstyle{remark} \newtheorem{rem}[defn]{Remark}
\newcommand{\be}{\begin{equation}}
\newcommand{\ee}{\end{equation}}

\renewcommand{\geq}{\geqslant}                                     %Always use \geqslant, never the standard \geq
\renewcommand{\leq}{\leqslant}                                     %Always use \leqslant, never the standard \leq
\title{Cullen's Stability Principle and Weak Solutions of the Free-surface Semi-geostrophic Equations}

\author{M. J. P. Cullen\thanks{Met Office, Exeter. E-mail: \href{mike.cullen@metoffice.gov.uk}{mike.cullen@metoffice.gov.uk}.}, \hspace{1mm} T. Kuna\thanks{Department of Mathematics and Statistics, University of Reading. E-mail: \href{t.kuna@reading.ac.uk}{t.kuna@reading.ac.uk}.}, \hspace{1mm} B. Pelloni$^{\ddagger}$ and \hspace{1mm}M. Wilkinson\thanks{Maxwell Institute for Mathematical Sciences and Department of Mathematics, Heriot-Watt University, Edinburgh. E-mails: \href{b.pelloni@hw.ac.uk}{b.pelloni@hw.ac.uk}, \href{mark.wilkinson@hw.ac.uk}{mark.wilkinson@hw.ac.uk}.}}

\date{}
\begin{document}
\maketitle

\begin{abstract}
\noindent The semi-geostrophic equations are used widely in the modelling of large-scale atmospheric flows. In this note, we prove the global existence of weak solutions of the incompressible semi-geostrophic equations, in geostrophic coordinates, in a three-dimensional domain with a free upper boundary. The proof, based on an energy minimisation argument originally inspired by Cullen’s Stability Principle, uses optimal transport results as well as the analysis of Hamiltonian ODEs in spaces of probability measures as studied by Ambrosio and Gangbo. We also give a general formulation of Cullen's Stability Principle in a rigorous mathematical framework.
\end{abstract}

\section{Introduction}\label{introduction}
The semi-geostrophic equations for an incompressible flow subject to a constant Coriolis force comprise the following system of equations for an unknown Eulerian velocity field $u$, geostrophic velocity field $u_{g}=(u_{g, 1}, u_{g, 2}, 0)$, pressure $p$, and density $\rho$,
%(see, for example, \cite{maroofi}):
\begin{equation}\label{sgincom}
\left\{
\begin{array}{l}
\displaystyle
D_t u^g + e _3 \wedge (u - u^g)= 0,
\vspace{2mm} \\
D_t \rho = 0,\vspace{2mm} \\
\nabla \cdot u=0,\vspace{2mm} \\
\nabla p = (u_2^g, -u_1^g, \rho),
\end{array}
\right.
\end{equation}
where $D_t$ denotes the material derivative operator, namely
 \be
D_t=\partial_t+ u\cdot\nabla.
\label{lagrder}\ee
In this note, we shall solve the above system formulated in so-called {\em geostrophic coordinates} in the time-dependent spatial domain $\Omega_{h_{t}}\subset \mathbb{R}^3$ given by 
\begin{equation*}
\Omega_{h_{t}}:= \{ (x_1, x_2, x_3) \in \mathbb{R}^3 : (x_1, x_2) \in B\hspace{2mm}\text{and}\hspace{2mm} 0 <x_3 < h( x_1, x_2,t) \},
\end{equation*}
with the upper free surface denoted by $\mathcal{S}_{t}$, namely
\begin{equation*}
\mathcal{S}_{t}:=\left\{
(x_{1}, x_{2}, h(x_{1}, x_{2}, t))\in\mathbb{R}^{3}\,:\,(x_{1}, x_{2})\in B
\right\},
\end{equation*}
where $B \subset \mathbb{R}^2$ is a fixed open bounded set to be considered as the base of the fluid domain, while $h$ is an unknown surface height function which characterises the free surface $\mathcal{S}_{t}$ in the absence of singularity formation, e.g. splashes or overturning crests. Moreover, in this work, the Eulerian velocity field $u$ is subject to the following flux free condition on the time-independent part of the boundary,
\begin{equation}\label{bc1}
u\cdot n =0 \quad \text{on}\hspace{2mm}\partial\Omega_{h_{t}}\setminus\mathcal{S}_{t},
\end{equation}
whilst it is subject to the kinematic boundary condition
\begin{equation}\label{bc2}
\frac{\partial h}{\partial t}+\sum_{j=1}^{2}u_{j}\frac{\partial h}{\partial x_{j}}=0
\end{equation}
on the free surface $\mathcal{S}_{t}$. Finally, the pressure $p$ is assumed to be constant on $\mathcal{S}_{t}$ for all times. For a model which includes more physics of the behaviour of large-scale atmospheric flows, one ought to consider the {\em compressible} semi-geostrophic equations together with a {\em variable} Coriolis parameter and a free upper boundary condition; we refer the reader to the monograph of Cullen (\cite{cullen}, chapter 4) for more information on the physics of semi-geostrophic flows. The mathematical complexity of this problem has meant that so far results in the literature have only been obtained after relaxing one or more of these conditions. We now give a brief summary of those mathematical results pertinent to system \eqref{sgincom} above.
\subsection{Brief Overview of the State of the Art}
In \cite{benamou}, Benamou and Brenier assumed the atmospheric fluid under study to be incompressible, the Coriolis parameter constant and the fluid domain to be fixed and independent of time. By regarding $\nabla p(\cdot, t)$ as a diffeomorphism for all times $t$, inspired by the original work of Hoskins in \cite{hoskins}, the authors derived what has been termed the {\em dual formulation} of the semi-geostrophic equations in so-called geostrophic co-ordinates that reveals the formal Hamiltonian structure of the dynamics. Indeed, in this formulation, the equations are interpreted as a Monge-Amp\`{e}re equation coupled with a transport equation corresponding to a time-dependent $\mathrm{BV}_{\mathrm{loc}}(\mathbb{R}^{3})$-vector field, and this elegant interpretation yields the proof of the existence of weak solutions of the system in geostrophic coordinates by way of the well-known Polar Factorisation Theorem of Brenier \cite{brenier}. 

This result was generalised by Cullen and Maroofi in \cite{maroofi} to prove the existence of weak solutions of the 3-dimensional {\em compressible} system, still under the assumption of a fixed fluid domain subject to a no-flux boundary condition.  

In \cite{gangbo}, Cullen and Gangbo relaxed the assumption of a rigid  fluid domain by assuming the more physically-appropriate free boundary condition for the incompressible system. However, the authors made the additional assumption of a constant potential temperature, and thus obtained a 2-dimensional system, known as the {\em semi-geostrophic shallow water system}, posed on a fixed two-dimensional domain. As such, in their formulation, the presence of a free surface was transformed away from the problem. After passing to variables in geostrophic coordinates, the authors proved the existence of weak solutions of the resulting problem in geostrophic coordinates.

All results above were obtained for the dual formulation of the equations in geostrophic coordinates, which is also the setting we consider in the present paper. However, we mention for completeness more recent results regarding the existence of solutions in Lagrangian coordinates. The first step in this direction was taken by Cullen and Feldman, who proved in \cite{feldman} the existence of Lagrangian solutions in physical variables for the rigid boundary case, a result that was extended in Cullen, Gilbert and Pelloni in \cite{CGP1} to the compressible system.  In a relatively-recent pair of works \cite{ambnew,ambnew2}, Ambrosio, Colombo, De Philippis and Figalli succedeed in constructing weak solutions of the semi-geostrophic equations in Eulerian co-ordinates for a small class of initial data.

\subsection{Contributions of this Article}
Motivated in part by the original unpublished paper of Cullen, Gilbert, Kuna and Pelloni \cite{CKGParx}, in his work \cite{cheng} Cheng recently proved the existence of Lagrangian weak solutions of the {\em incompressible system, in three-dimensional space, in a domain with a free upper boundary}. He also gave a direct proof, without appealing to the work of Ambrosio and Gangbo \cite{ambgangbo}, of the existence of weak solutions of the free-surface system in geostrophic coordinates. The latter result was first announced in the PhD thesis of Gilbert \cite{dkgthesis}. In this work, we give a more concise presentation of the results in \cite{CKGParx} and introduce some novel elements:
 \begin{itemize}
 \item
To the knowledge of the authors, for the first time, we give a rigorous mathematical formulation of {\em Cullen's Stability Principle} for solutions of the semi-geostrophic equations. We do this by making use of the notion of {\em inner variation} of an energy functional (see Giaquinta and Hildebrandt \cite{giaquinta2013calculus}). Using this, we offer a precise definition of stable weak solution of the free-surface system;
\item We introduce the notion of {\em $\mathcal{A}$-Hamiltonian} which allows us to develop a general theory of free-surface problems when the source measure in the Monge-Kantorovich problem is an unknown;
\item We obtain our proof of the main existence result (See Theorem \ref{new5.5} below) by employing the general theory of Hamiltonian ODEs in Wasserstein spaces of probability measures due to Ambrosio and Gangbo in \cite{ambgangbo}. The strategy of the proof is to show that the Hamiltonian of the system, given by the geostrophic energy, satisfies the conditions necessary to invoke the general theory of \cite{ambgangbo}. This approach is more direct than that taken in \cite{cheng}, in which the author constructs a dynamics by way of a time-stepping algorithm `by hand'.
 \end{itemize}
 
\subsection{Main Result of this Article}
The main result states the existence of solutions for system \eqref{sgincom} formulated in geostrophic coordinates, namely equation \eqref{fun} below. We state this for initial surface profiles assumed to be in a space of Lipschitz continuous functions. In particular, the solutions we construct have the property that the free surface $\mathcal{S}_{t}$ is for all times the graph of a Lipschitz function. 

\begin{theorem}\label{new5.5}

Let $1 \leq p \leq \infty $, and let $\mathcal{P}^{2}_{ac}(\mathbb{R}^{3})\ni \nu _0\ll \mathscr{L}^{3}$ with a compactly-supported density in $L^p (\mathbb{R}^{3})$ be given. Let also $h_0\in W^{1,\infty}(B)$ be given, and be compatible with $\nu_{0}$ in the sense that
\begin{equation*}
h_{0}=\underset{h\in\mathcal{A}_{\ast}}{\mathrm{arg\,min}}\left(
\min_{T\in\mathcal{T}(\sigma_{h}, \nu_{0})}e(x, T(x))\,dx
\right),
\end{equation*}
where $\mathcal{A}_{\ast}$ is defined in \eqref{astar} below. It follows that the free-surface semi-geostrophic equations in geostrophic coordinates given by
\begin{equation}\label{gcfirst}
\left\{
\begin{array}{l}
\partial_{t}\nu+\nabla\cdot(J(\mathrm{id}_{\mathbb{R}^{3}}-\nabla P^{\ast})\nu)=0,\vspace{2mm}\\
\mathrm{det}\,D^{2}P^{\ast}=\nu,
\end{array}
\right.
\end{equation}
admits a corresponding stable\footnote{The precise definition of stable global-in-time weak solution of this system is given in \ref{astabweak} below.} global-in-time weak solution $(h, \nu)$.
\end{theorem}
We bring to the attention of the reader that while one would expect to furnish the initial-value problem associated to \eqref{sgincom} with $h_{0}$ and $\nabla P_{0}$, one rather furnishes the initial-value problem for \eqref{gcfirst} with $\nu_{0}$ alone. The reason for this will hopefully become clear to the reader in our formulation of the problem in geostrophic coordinates (see section \ref{eulerform} below). Moreover, the precise definition of {\em stable} global-in-time weak solution of \eqref{gcfirst} will be offered in section \ref{cullenrigour} below.
\subsection{Notation}
In all that follows: $\wedge$ denotes the exterior product on $\mathbb{R}^{3}$; $\mathscr{L}^{3}$ denotes the Lebesgue measure on $\mathbb{R}^{3}$; $\nu\ll\mathscr{L}^{3}$ denotes that a measure $\nu$ is absolutely continuous with respect to $\mathscr{L}^{3}$, while $\mathscr{L}_{X}$ denotes the restriction of $\mathscr{L}^{3}$ to a measurable subset $X\subset\mathbb{R}^{3}$; $\mathcal{P}^{2}_{ac}(\mathbb{R}^{3})$ denotes the set of all probability measures $\nu$ on $\mathbb{R}^{3}$ which admit the property $\nu\ll\mathscr{L}^{3}$ and which have a finite second moment on $\mathbb{R}^{3}$; $\mathrm{id}_{X}$ denotes the identity map $x\mapsto x$ on a set $X\subseteq\mathbb{R}^{3}$; if $T=T(x, t)$ is a space-time map, the we write $T_{t}:=T(\cdot, t)$; if $X\subseteq\mathbb{R}^{3}$ is an open set, then for each $k\in \{1, 2, 3, ...\}$, $\mathrm{Diff}^{k}(X)$ denotes the set of all diffeomorphisms of class $\mathcal{C}^{k}$ on X, while $\mathrm{Diff}(X)$ denotes the set of all infinitely-differentiable diffeomorphisms on $X$; $\mathbb{R}^{3\times 3}_{+}$ denotes the set of $3\times 3$ positive semi-definite matrices with real entries; $W^{k, p}(X)$ denotes the Sobolev space of distributionally-differentiable maps on $X$ with smoothness $k$ and integrability $p$. Finally, in all the sequel, we identify any measure which is absolutely continuous (w.r.t. $\mathscr{L}^{3}$) with its corresponding density.
\section{Derivation of the Free Surface Semi-Geostrophic Equations in Geostrophic Coordinates}\label{eulerform}
In this section, we derive the free surface semi-geostrophic equations in geostrophic coordinates which we shall study in all the sequel. We do so by following the original approach of Benamou and Brenier (\cite{benamou}, section 2.2) in the present more complicated case of a time-dependent free surface $\mathcal{S}_{t}$. The main difficulty in performing this derivation is that, unlike in \cite{benamou}, the source measure in a certain important Monge-Kantorovich problem is an unknown when the fluid domain can vary with time, and this point requires some careful discussion.
\subsection{Alternative Eulerian Formulation of the Equations}
The system \eqref{sgincom} is a formulation of semi-geostrophic dynamics in Eulerian coordinates for a fluid with Eulerian velocity profile $u$. However, as the reader will note, there is no explicit time evolution equation for $u$. 
An alternative Eulerian formulation in terms of a conservative vector field can be obtained by defining a {\em modified pressure} $P$ defined pointwise as
\begin{equation*}
P(x, t):=p(x, t)+\frac 12\left(x_{1}^{2}+x_{2}^{2}\right) \quad \text{for}\hspace{2mm}x\in \Omega_{h_{t}}.
\end{equation*}
Equations (\ref{sgincom}) can then be written in the equivalent form given by
\begin{equation}\label{saag}
\left\{
\begin{array}{l}
\displaystyle D_t\nabla P=J(\nabla P-\mathrm{id}_{\Omega_{h_{t}}})\quad \text{on}\hspace{2mm}\Omega_{h_{t}}, \vspace{2mm} \\
\nabla\cdot u=0\quad \text{on}\hspace{2mm}\Omega_{h_{t}}, 
\end{array}%\tag{SG}
\right.
\end{equation}
where the matrix $J\in\mathbb{R}^{3\times 3}$ is given by 
\begin{equation*}
J=\left(\begin{array}{ccc}
0&-1&0
\\
1&0&0
\\0&0&0
\end{array}\right).
\end{equation*}
The semi-geostrophic equations now read as a semi-linear transport equation in an unknown time-dependent conservative vector field, namely $\nabla P$. The unknown Eulerian velocity field $u$ may be determined by way of the nonlinear condition that it advect $\nabla P$ whilst preserving its property of conservativeness. As originally observed by Cullen and Purser \cite{purser}, it is physically meaningful to solve system \eqref{saag} only for those conservative vector fields which are the gradient of a time-dependent convex-in-space function. We shall return to the importance of convexity in section \ref{cullenrigour} below.
\subsection{Classical Solutions in Eulerian Coordinates}
We begin by stating our definition of classical solution of the initial-value problem for \eqref{sgincom}. To do so, we must first state what we mean by smoothness of maps on a time-dependent graph domain. 
\begin{defn}
Suppose $\tau>0$ and $k\in\{0, 1, 2, ...\}$ are given. Let $B\subset\mathbb{R}^{2}$ be an open set with boundary $\partial B$ of class $C^{k}$. If $h\in C^{k}(B\times (0, T))$ is a given non-negative function, we write ${\cal D}_{h}$ to denote the associated open subset of $\mathbb{R}^{4}$ given by
\begin{equation*}
{\cal D}_{h}:=\bigcup_{t\in (0, \tau)}\Omega_{h_{t}}\times \{t\}.
\end{equation*}
For the given $h$, we say that a one-parameter family of functions 
\begin{equation*}
\mathcal{F}:=\left\{
f_{t}\,:\,0< t< \tau
\right\} \quad \text{with}\hspace{2mm}f_{t}:\Omega_{h_{t}}\rightarrow\mathbb{R}
\end{equation*}
is of class ${\cal C}^{k}$ (on $\mathcal{D}_{h}$) if and only if the map $(x, t)\mapsto f_t(x)$ belongs to $C^{k}({\cal D}_{h})$.
\end{defn}
We are now able to offer the following definition of classical solution of the free surface semi-geostrophic equations.
\begin{defn}[Global-in-time Classical Solutions in Eulerian Coordinates]\label{gitcs}
Suppose given initial data
\begin{equation*}
h_{0}\in C^{1}(B)\cap C^{0}(\overline{B})\quad \text{and} \quad P_{0}\in C^{2}(\Omega_{h_{0}})\cap C^{0}(\overline{\Omega_{h_{0}}})
\end{equation*}
satisfying $
h_{0}> 0$ on $\overline{B}$
and
\begin{equation*}
P_{0}(x_{1}, x_{2}, h_{0}(x_{1}, x_{2}))=\frac 1 2 \left( x_1^2+x_2^2\right)\quad \text{for all}\hspace{2mm}(x_{1}, x_{2})\in\overline{B}.
\end{equation*}
We say that the triple $(h, P, u)$ is a {\bf global-in-time classical solution} of \eqref{saag} for the given data $(h_{0}, P_{0})$ if and only if for every $\tau>0$, one has that
\begin{equation*}
h\in C^{1}(B\times (0, \tau)) \quad \text{with}\hspace{2mm}h>0\hspace{2mm}\text{on}\hspace{2mm}B\times (0, \tau),
\end{equation*}
the maps $P$ and $u$ are of class ${\cal C}^{2}$ and ${\cal C}^{1}$ on ${\cal D}_{h}$, respectively, and satisfy the equations \eqref{saag} and the boundary conditions \eqref{bc1}--\eqref{bc2} pointwise in the classical sense. In particular, the solution preserves the (modified) surface pressure condition
\begin{equation*}
P(x_{1}, x_{2}, h(x_{1}, x_{2}, t), t)=\frac 1 2 \left(x_1^2+x_2^2\right), \quad \text{for all}\hspace{2mm}(x_{1}, x_{2})\in \overline{B}\hspace{2mm}\text{and}\hspace{2mm}0\leq t\leq\tau,
\end{equation*}
and the triple $(h, P, u)$ is compatible with the initial data in the sense that
\begin{equation*}
h(x_{1}, x_{2}, 0)=h_{0}(x_{1}, x_{2})\quad \text{for all}\hspace{2mm}(x_{1}, x_{2})\in B
\end{equation*}
and
\begin{equation*}
\nabla P(x, 0)=\nabla P_{0}(x) \quad \text{for all}\hspace{2mm}x\in\Omega_{h_{0}}.
\end{equation*}
\end{defn}
The definition of a global-in-time classical solution for the initial-value problem associated to \eqref{saag} is immediate from the above. Owing to the dearth of techniques which would allow one to construct smooth solutions in Eulerian coordinates, we instead aim to construct solutions of \eqref{sgincom} in a different and more mathematically-amenable coordinate system.
\subsection{Formulation in Geostrophic Coordinates}
Let us suppose that $(h, P, u)$ is a global-in-time classical solution of \eqref{sgincom} with the additional property that $x\mapsto \nabla P(x, t)$ is smooth diffeomorphism for all times $t$. Following \cite{benamou}, by taking Euclidean inner products throughout \eqref{saag} with $\nabla \xi(\nabla P(x, t), t)$ for any $\xi$ whose associated family $\{\xi_{t}\}_{t>0}$ is of class $\mathcal{C}^{\infty}$ on $\mathcal{D}_{h}$, one can show that
\begin{equation*}
\int_{0}^{\tau}\int_{\nabla P_{t}(\Omega_{h_{t}})}\left(\partial_{t}+(U\cdot\nabla)\right)\xi(X, t)\nu(X, t)\,dXdt=0
\end{equation*}
holds true, where 
\begin{equation*}
U:=J(\mathrm{id}_{\mathbb{R}^{3}}-\nabla P^{\ast})
\end{equation*}
and $\nu$ is defined pointwise as
\begin{equation*}
\nu(X, t):=\mathrm{det}\,D^{2}P^{\ast}(X, t),
\end{equation*}
where $P^{\ast}(\cdot, t)$ denotes the Legendre-Fenchel transform of $P(\cdot, t)$ on the open set $\Omega_{h_{t}}$. In other words, one has that the quantities $\nabla P^{\ast}$ and $\nu$ satisfy the coupled system
\begin{equation}\label{fun}
\left\{
\begin{array}{l}
\partial_{t}\nu+\mathrm{div}(U\nu)=0, \vspace{2mm}\\
\mathrm{det}\,D^{2}P^{\ast}=\nu
\end{array}
\right.
\end{equation}
pointwise in the classical sense on $\nabla P_{t}(\Omega_{h_{t}})$ for each time $t$. As such, this calculation motivates the following definition of global-in-time weak solution of system \eqref{fun}.
\begin{defn}[Global-in-time Weak Solutions in Geostrophic Coordinates]
Suppose $\nu_{0}\in\mathcal{P}^{2}_{ac}(\mathbb{R}^{3})$ and $h_{0}\in W^{1, \infty}(B)$ are given. We say that $(h, \nu)$ is a {\bf global-in-time weak solution} of \eqref{fun} if and only if for any $\tau>0$, one has that
\begin{equation*}
h\in C^{0}([0, \tau); W^{1, \infty}(B))
\end{equation*}
and the map $\nu:[0, \tau)\rightarrow \mathcal{P}^{2}_{ac}(\mathbb{R}^{3})$ is absolutely continuous and satisfies the transport equation in weak form given by
\begin{equation*}
\int_{0}^{\infty}\int_{\mathbb{R}^{3}}\left(\partial_{t}\phi(X, t)+(U\cdot\nabla_{X})\phi(X, t)\right)\nu(X, t)\,dXdt=0
\end{equation*}
for all $\phi\in C^{\infty}(\mathbb{R}^{3}\times (0, \tau))$ for which $\{\phi_{t}\}_{t>0}$ is of class $\mathcal{C}^{\infty}$ on $\mathcal{D}_{h}$. Moreover, for a.e. $t\in (0, \tau)$, the map $P(\cdot, t)$ is a Brenier solution of second boundary value problem for the Monge-Amp\`{e}re equation
\begin{equation*}
\mathrm{det}\,D^{2}P^{\ast}_{t}=\nu_{t}
\end{equation*}
with $\nabla P_{t}(\Omega_{h_{t}})=\mathrm{supp}\,\nu_{t}$.
\end{defn}
In the case of the semi-geostrophic equations in a fixed fluid domain, the second boundary value problem for the Monge-Amp\`{e}re equation in \eqref{fun} is fully determined, in the sense that the source domain $\Omega$ (and therefore the source measure $\mathscr{L}_{\Omega}$) is known and fixed for all times. However, in the present study of the free surface semi-geostrophic equations in geostrophic coordinates, the source domain $\Omega_{h_{t}}$ is an {\em unknown} of the problem for each $t$. Indeed, due to the absence of an evolution equation for the velocity field $u$ in geostrophic coordinates, there is no immediately obvious way by which to determine the free surface function $h$. Understanding how to solve the second BVP for the Monge-Amp\`ere equation in the case of a free surface is one of the contributions of this paper. Our means by which to do this is a careful study of Cullen's Stability Principle.
\subsection{The Geostrophic Energy Functional}
The geostrophic energy, which is constant in time when composed with smooth solutions of the semi-geostrophic equations, is given by the functional
\begin{equation}\label{bbenergy}
E_{\Omega_{h_{t}}}[\nabla p] = \int_{\Omega_{h_{t}}} \! \bigg( \frac{1}{2}((u_1^g)^2 + (u_2^g)^2) + \rho x_3 \bigg)\, dx,
\end{equation}
where $dx$ denotes the restriction of the Lebesgue measure $\mathscr{L}^{3}$ to the open set $\Omega_{h_{t}}$. Formally, a smooth solution of the system is to be regarded as {\em stable} if and only if it admits the following property: 
\begin{principle}[Cullen's Stability Principle]
Stable solutions of \eqref{saag} are those which, at each fixed time $t$, minimise the energy given by (\ref{bbenergy}) with respect to the rearrangements of particles, in physical space, that conserve the absolute momentum $(u_1^g-x_2, u^g_2+x_1)$ and the density $\rho$.
\label{princ1}\end{principle}
Of course, what constitute `rearrangements of particles' is yet to be specified in precise mathematical terms. This principle was expressed in  \cite{shutts} as the requirement that those flows corresponding to critical points of (\ref{bbenergy}) with respect to such constrained rearrangements of particles in physical space are precisely those flows in hydrostatic and geostrophic balance. By way of some formal calculations (namely \cite{cullen}, section 3.2), Cullen has shown principle 2.1 formally to be equivalent to the following:
\begin{principle}[Cullen's Convexity Principle]
Minima of the energy (\ref{bbenergy}) correspond to a modified pressure $P(x, t)$ which is a convex function of $x$. 
\end{principle} 
Whilst having been very successful in leading one's intuition and understanding of the system, and certainly being of importance when analysing system \eqref{saag} with the tools of elliptic PDE theory, these principles have not been expressed in the literature in precise mathematical terms. We do this below by appealing to the notion of {\em inner variation} from the general theory of the calculus of variations.
\subsection{Mathematical formulation of Cullen's Stability Principle}\label{cullenrigour}
In what follows, we employ the following notion of {\em inner variation} of an energy functional (see Giaquinta and Hildebrandt \cite{giaquinta2013calculus}, section 3.1).
\begin{defn}[First Inner Variation]\label{firstvar}
Suppose $X\subset\mathbb{R}^{3}$ is an open bounded set, and $e\in C^{1}(\mathbb{R}^{3}, \mathbb{R})$. For $k\geq 1$ and some $1\leq p\leq \infty$, consider the functional $E:\mathrm{Diff}_{k}(X)\rightarrow\mathbb{R}$ defined by
\begin{equation*}
E[T]:=\int_{X}e(T(x))\,dx \quad \text{for}\hspace{2mm}T\in L^{p}(X).
\end{equation*}
If the limit exists, we say that $\langle \delta E[T], \Phi\rangle\in\mathbb{R}$ defined by 
\begin{equation*}
\langle\delta E[T], \Phi\rangle:=\lim_{\varepsilon\downarrow 0}\frac{E[T\circ \Phi_{\varepsilon}]-E[T]}{\varepsilon} \quad \text{for}\hspace{2mm}\Phi\in\mathrm{Diff}_{k}(X)
\end{equation*}
is the inner variation of $E$ at $T$ in the direction $\Phi\in\mathrm{Diff}_{k}(X)$, where $\Phi_{\varepsilon}:=\mathrm{Id}+\varepsilon\Phi$ for $\varepsilon>0$.
\end{defn}
In a natural manner, we can also define the notion of {\em second inner variation} of an energy functional.
\begin{defn}[Second inner variation]
Under the same conditions of definition \ref{firstvar}, if the limit exists, we say that $\langle\delta^{2}E[T]; \Psi, \Phi\rangle\in\mathbb{R}$ defined by
\begin{equation*}
\langle\delta^{2}E[T]; \Phi, \Psi\rangle:=\lim_{\varepsilon\downarrow 0}\frac{\langle\delta E[T\circ \Psi_{\varepsilon}], \Phi \rangle-\langle\delta E[T], \Phi\rangle}{\varepsilon}
\end{equation*}
for $\Psi, \Phi\in \mathrm{Diff}_{k}(X)$ is the second inner variation of $E$ at $T$ in the direction $(\Phi, \Psi)\in\mathrm{Diff}_{k}(X)\times\mathrm{Diff}_{k}(X)$, where $\Psi_{\varepsilon}:=\mathrm{Id}+\varepsilon\Psi$ for $\varepsilon>0$.
\end{defn}
Using the second inner variation as given above, one can define a notion of stable weak solution of \eqref{fun}. The following definition constitutes a rigorous reformulation of {\em Cullen's Stability Principle} (\ref{princ1})
\begin{defn}[Stable Weak Solutions of \eqref{fun}]\label{stabledef}
Let $(h, \nu)$ be a global-in-time weak solution of \eqref{fun} corresponding to given initial data $(h_{0}, \nu_{0})$. We say that $(h, \nu)$ is \textbf{stable} if and only if for each time $t\geq 0$, it holds that the energy functional $E_{\Omega_h}$ defined by \eqref{bbenergy} satisfies
\begin{equation*}
\langle E_{\Omega_{h_{t}}}[\nabla P_{t}], \Phi \rangle=0
\end{equation*}
and
\begin{equation*}
\langle\delta^{2}E_{\Omega_{h_{t}}}[\nabla P_{t}]; \Phi, \Phi\rangle\geq 0
\end{equation*}
for all $\Phi\in C^{\infty}_{c}(\Omega_{h_{t}}, \mathbb{R}^{3})$.
\end{defn}
At this point, it is far from obvious how one can construct a solution of the active transport equation which admits this stability property. The following useful proposition provides a necessary and sufficient condition for stability of weak solutions.
\begin{prop}
A global-in-time weak solution $(h, \nu)$ of \eqref{fun} with the property that $\nabla P(\cdot, t)\in W^{1, p}(\Omega_{h_{t}})$ for a.e. $t\in (0, \infty)$ and some $1\leq p\leq \infty$ is stable if and only if $\nabla P(\cdot, t)$ is a critical point of $E_{\Omega_{h_{t}}}$ with the property that $P_{t}$ is convex on $\Omega_{h_{t}}$ for almost every $x\in\Omega_{h_{t}}$.
\end{prop}
\begin{proof}
For any $\Psi\in\mathrm{Diff}(\Omega_{h_{t}}, \mathbb{R}^{3})$, one has that
\begin{equation*}
\begin{array}{l}
\displaystyle \left\langle \frac{\delta E_{\Omega_{h_{t}}}[\nabla P_{t}\circ \Psi_{\varepsilon}]- \delta E_{\Omega_{h_{t}}}[\nabla P_{t}\circ \Psi_{0}], \Phi}{\varepsilon}\right\rangle= \int_{\Omega_{h_{t}}}\left(
\frac{\nabla P_{t}(\Psi_{\varepsilon}(x))-\nabla P_{t}(x)}{\varepsilon}
\right)\cdot\Phi(x)\,dx,
\end{array}
\end{equation*}
from which it follows that
\begin{equation*}
\langle \delta^{2}E_{\Omega_{h_{t}}}[\nabla P_{t}]; \Phi, \Psi\rangle=\int_{\Omega_{h_{t}}}\Phi(x)\cdot D^{2}P(x, t)\Psi(x)\,dx.
\end{equation*}
Thus, it follows from definition \ref{stabledef} above that $(h, \nu)$ is a stable global-in-time weak solution of \eqref{fun} if and only if $\nabla P(\cdot, t)$ is a critical point of $E_{\Omega_{h_{t}}}$ such that $P_{t}$ is convex on $\Omega_{h_{t}}$ for almost every $x\in\Omega_{h_{t}}$.
\end{proof}
We infer from this proposition that if for a.e. time $t$, the map $\nabla P_{t}$ minimises the geostrophic energy $E_{\Omega_{h_{t}}}$ in some suitable class of vector fields $\mathcal{E}[h_{t}]$ containing $\mathrm{Diff}(\Omega_{h_{t}})$, then the associated weak solution $(h, \nabla P, u)$ of \eqref{fun} is stable.

It was observed by Benamou and Brenier in the original study of \eqref{fun} posed on a fixed bounded fluid domain $\Omega\subset\mathbb{R}^{3}$ that a natural means by which to ensure convexity of the geopotential $P(\cdot, t)$ is to treat the minimisation of the geostrophic energy as a {\em Monge} (or, equivalently, in the case we deal with the cost function $e$ in \eqref{cost} below, a {\em Monge-Kantorovich}) problem.  Indeed, it follows from the well-known work of Brenier that $e$-optimal maps are the realised as the gradient of convex functions, where $e:\mathbb{R}^{3}\times \mathbb{R}^{3}\rightarrow\mathbb{R}$ is the cost function, equivalent to the the classical squared Euclidean cost, given by
\begin{equation}\label{cost}
e(x, y):=\frac{1}{2}\left((x_1-y_1)^{2}+(x_2-y_2)^{2}\right)-x_3y_3
\end{equation}
for $x, y\in\mathbb{R}^{3}$. The following proposition is immediate from the above.
\begin{prop}
A global-in-time weak solution $(h, \nu)$ of \eqref{fun} with the property that $\nabla P_{t}\in W^{1, p}(\Omega_{h_{t}})$ for a.e. $t$ is stable if and only if for a.e. $t$, $\nabla P_{t}$ is the $e$-optimal map from the source measure $\mathscr{L}_{\Omega_{h_{t}}}$ to the target measure $\nabla P_{t}\#\mathscr{L}_{\Omega_{h_{t}}}$, i.e.
\begin{equation*}
\nabla P_{t}=\underset{T\in \mathcal{T}(\mathscr{L}_{\Omega_{h_{t}}}, \nabla P_{t}\#\mathscr{L}_{\Omega_{h_{t}}})}{\mathrm{argmin}}\,E_{\Omega_{h_{t}}}[T],
\end{equation*}
where $\mathcal{T}(\mu, \nu)$ denotes the set of all transport plans from $\mu$ to $\nu$.
\end{prop}
\begin{rem}
We refer the reader to the monograph of Villani \cite{villanioldnew} for basic concepts in the theory of optimal transport.
\end{rem}
It is now that we face our first major difficulty in the construction of weak solutions of the free-surface semi-geostrophic system in geostrophic coordinates \eqref{fun}. To have a well-posed Monge or Monge-Kantorovich problem, one needs to provide both the source and target measure. In the current formulation of the problem in geostrophic coordinates, there is no way by which to determine the free surface function $h$, and so the source measure $\mathscr{L}_{\Omega_{h_{t}}}$ is unknown. We now develop a general framework in which one may consider free-surface semi-geostrophic dynamics in geostrophic coordinates.
\subsection{Determination of the Source Measure}\label{detsource}
The idea to which we shall appeal in the sequel (which is consistent with Cullen's Stability Principle as stated in Cullen and Gangbo \cite{gangbo}) is that for each time $t$ it is not only the geopotential $P(\cdot, t):\Omega_{h_{t}}\rightarrow\mathbb{R}$ that ought to be stable (in the sense of definition \ref{stabledef} above), but the free surface $\mathcal{S}_{t}$ should enjoy some kind of `natural' stability property as well. In rough terms, the notion of stability we employ is that a weak solution $(h, \nu)$ ought to have the property that for almost all times the free surface profile $h_{t}$ minimises the functional
\begin{equation}\label{funky}
\eta\mapsto \inf_{\gamma\in\Gamma(\mathscr{L}_{\Omega_{\eta}}, \nu_{t})}\int_{\Omega_{\eta}\times\mathbb{R}^{3}}e(x, y)\,d\gamma(x, y)
\end{equation}
over some appropriate class of surface profiles $\eta:\Omega\rightarrow [0, \infty)$, where the absolutely continuous measure $\sigma_{\eta}\in\mathcal{P}^{2}_{ac}(\mathbb{R}^{3})$ associated to the surface profile $\eta$ is defined to be
\begin{equation*}
\sigma_{\eta}:=\mathds{1}_{H}\mathscr{L}^{3}, \quad \text{where}\hspace{2mm}H:=\{(x_{1}, x_{2}, \eta(x_{1}, x_{2}))\in\mathbb{R}^{3}\,:\,(x_{1}, x_{2})\in B\}.
\end{equation*}
Defining a functional by the minimisation requirement \eqref{funky} gives rise to a double minimisation problem, in which the inner minimisation constitutes a Monge problem (or a Monge-Kantorovich problem), while the outer minimisation is to be tackled using techniques of the calculus of variations. To define our notion of stability of free surface dynamics rigorously, we are required to specify the class of profiles $\eta$ in which the minimisation is considered. We do this now.
\begin{defn}[Admissible Class of Fluid Domains]
A non-empty class ${{\cal{A}}}\subseteq 2^{\mathbb{R}^{3}}$ of subsets is said to be {\bf admissible} if and only if it has the following properties:
\begin{enumerate}
\item each $A\in{\cal{A}}$ admits the representation
\begin{equation*}
A=\Omega_\eta:=\left\{
(x_{1}, x_{2}, x_{3})\in\mathbb{R}^{3}\,:\,(x_{1}, x_{2})\in B\hspace{2mm}\text{and}\hspace{2mm}0<x_{3}<\eta(x_{1}, x_{2})
\right\},
\end{equation*}
for some $\eta:\Omega_{2}\rightarrow[0, \infty]$ which is of class $L^{1}(B)$;
\item each $A\in{\cal{A}}$ is of unit mass, i.e. $\mathscr{L}^{3}(A)=1$ for all $A\in{\cal{A}}$.
\end{enumerate}
\end{defn}
In our context, one interprets each admissible class of sets ${\cal{A}}$ as the collection of all possible configurations that the free surface geostrophic fluid can assume during its motion. It will be particularly convenient in what follows to generate classes of admissible sets by  functional spaces. 
\begin{defn}
Suppose ${\cal Y}$ is a non-empty subset of non-negative maps contained the unit sphere of $L^{1}(\Omega)$. We say that an admissible class of sets ${\cal{A}}$ is {\bf generated by} ${\cal Y}\subset L^{1}(\Omega)$ if and only if
\begin{equation*}
{\cal{A}}=\left\{
\Omega_{\eta}\,:\,\eta\in {\cal Y}
\right\}.
\end{equation*}
\end{defn}
We shall also employ the notation $\mathcal{F}_{{\cal{A}}}$ to denote the class of characteristic functions associated to members of ${\cal{A}}$, namely
\begin{equation*}
\mathcal{F}_{{\cal{A}}}:=\left\{
\mathds{1}_{A}\in L^{\infty}(B\times [0, \infty))\,:\,A\in{\cal{A}}
\right\}
\end{equation*} 
With these definitions in place, we define one of the fundamental object of interest in this article. Indeed, we shall focus our efforts on its analysis in the rest of our work.
\begin{defn}[${\cal{A}}$-Hamiltonian]
Let $\mathcal{A}$ be an admissible class of fluid domains. The associated {\bf ${\cal{A}}$-Hamiltonian} $H_{{\cal{A}}}:\mathcal{P}^{2}_{ac}(\mathbb{R}^{3})\rightarrow[-\infty, \infty]$ is defined by 
\begin{equation}\label{Aham}
H_{{\cal{A}}}(\nu):=\inf_{\sigma\in\mathcal{F}_{{\cal{A}}}}\left(\inf_{\gamma\in\Gamma(\sigma\mathscr{L}^{3},\, \nu)}\int_{\mathbb{R}^{3}\times\mathbb{R}^{3}}e(x, y)\,d\gamma(x, y)\right),
\end{equation}
for $\nu\in\mathcal{P}^{2}_{ac}(\mathbb{R}^{3})$ and $e$ defined by (\ref{cost}).
\end{defn}
\begin{rem}
The use of the term {\em Hamiltonian} will be fully justified in section \ref{newmain} below, when we understand $H_{\mathcal{A}}$ as giving rise to a smooth map on a Wasserstein metric space of probability measures for well-chosen $\mathcal{A}$.
\end{rem}
The specific properties of a given ${\cal{A}}$-Hamiltonian $H_{{\cal{A}}}$ depend on the chosen admissible class ${\cal{A}}$. Therefore the admissible class ${\cal{A}}$ should be viewed as a datum of the problem, thereby becoming {\em a part of the model to be chosen appropriately}. In this article we shall work with one `natural' choice of ${\cal{A}}$, as studied in \cite{cheng}, namely all those generated by bounded continuous functions on $B$. 

We are now in a position to define a notion of stable weak solution of the incompressible free surface semi-geostrophic equations in geostrophic coordinates which depends on the choice of class $\mathcal{A}$ of free surface profiles.
\begin{defn}[${\cal{A}}$-stable Global-in-time Weak Solution of (\ref{fun})]\label{astabweak}
Suppose that $\mathcal{A}$ is an admissible class of fluid domains. We say that a global-in-time weak solution $(h, \nu)$ of is {\bf ${\cal{A}}$-stable} if and only if for a.e. $t\geq 0$, one has that
\begin{equation*}
\int_{\Omega_{h_{t}}}e(x, \nabla P(x, t))\,d\sigma_{h_{t}}(x)=H_{{\cal{A}}}(\nu_{t}).
\end{equation*}
\end{defn}
The above definition makes it clear that if ${\cal{A}}$ is well chosen, one can expect the stability criterion \ref{astabweak} to select, in a unique manner, the form of the free surface at each time. In this framework, the second BVP for the Monge-Amp`{e}re equation is fully determined, and it is in turn possible to construct a global-in-time weak solution of the free surface semi-geostrophic equations expressed in geostrophic coordinates. 
\section{Proof of the Main Theorem}
In all that follows, we work with the particular $\mathcal{A}$-Hamiltonian which corresponds to an admissible $\mathcal{A}_{\ast}$ which is generated by a class of continuous functions, namely
\begin{equation}\label{astar}
\mathcal{A}_{\ast}:=\left\{
\Omega_{\eta}\subset\mathbb{R}^{3}\,:\,\eta\in C^{0}(\overline{B}), \hspace{2mm}\eta\geq 0\hspace{2mm}\text{and}\hspace{2mm}\int_{B}\eta=1
\right\}.
\end{equation}
In this section, we draw upon the results of Cheng which aid in showing that $H_{\mathcal{A}_{\ast}}$ may be considered as a Hamiltonian on a Wasserstein metric space of probability measures, following Ambrosio and Gangbo. Indeed, we quote the following result which is essentially contained in \cite{cheng}.   
\begin{prop}[Cheng (2016)]\label{chengres}
Suppose $\nu\in\mathcal{P}^{2}_{ac}(\mathbb{R}^{3})$ with compact support is given. The functional
\begin{equation}\label{surfacefunk}
\eta\mapsto \inf_{T\in \mathcal{T}(\sigma_{\eta}, \nu)}\int_{\mathbb{R}^{3}}e(x, T(x))\,d\sigma_{\eta}(x)
\end{equation}
admits a minimum over the class $\mathcal{A}_{\ast}$ which is realised by a unique $h\in\mathcal{A}_{\ast}$. The assignment $\nu\mapsto h$ is continuous as a map from $\mathcal{P}^{2}_{ac}(\mathbb{R}^{3})$ (endowed with the narrow topology) to $L^{\infty}(B)$. 
\end{prop}
\begin{proof}
This follows from \cite{cheng}, corollary 2.11 and theorem 2.16.
\end{proof}
\subsection{Construction of a Free-Surface Dynamics}\label{newmain}
In this section we prove our main result, namely Theorem \ref{new5.5}. As mentioned above, we shall make use of the theory of Hamiltonian ODE of \cite{ambgangbo} in order to construct global-in-time weak solutions of system \eqref{fun}. We refer the reader to that work for basic definitions (such as the Fr\'{e}chet subdifferential of a map on $\mathcal{P}^{2}_{ac}(\mathbb{R}^{3})$, or $\lambda$-convexity). Let us begin with some definitions. 
\begin{defn}[Hamiltonian on $\mathcal{P}^{2}_{ac}(\mathbb{R}^{3})$]\label{hammy}
We say that a map $H:\mathcal{P}^{2}_{ac}(\mathbb{R}^{3})\rightarrow \mathbb{R}$ is a Hamiltonian on $\mathcal{P}^{2}_{ac}(\mathbb{R}^{3})$ if and only if for any $\nu_{0}\in \mathcal{P}^{2}_{ac}(\mathbb{R}^{3})$, it admits the following three properties:\\*[3mm]
{\bf (H1)} There exist associated constants $C_0=C_{0}(\nu_{0}) \in (0, \infty )$ and $R_0=R_{0}(\nu_{0}) \in (0, \infty ]$ such that for all $\nu \in \mathcal{P}_{ac}^2(\mathbb{R}^3)$ with $W_2(\nu , \nu _0) < R_0$, one has $\nu\in D(H)$, $\partial H(\nu)\neq \varnothing$, and $w :=\nabla H(\nu)$ satisfies $|w(y)| \leqslant C_0(1 + |y|)$ for $\nu$-a.e. $y\in \mathbb{R}^3$.\\*[3mm]
{\bf (H2)} If for $\nu \in \mathcal{P}_{ac}^2(\mathbb{R}^3)$ and $\{\nu_{j}\}_{j=1}^{\infty}\subset\mathcal{P}^{2}_{ac}(\mathbb{R}^{3})$ one has $\sup_{j}W_2(\nu_{j}, \nu_0)<R_0$ and $\nu_{j} \rightarrow \nu $ in the narrow topology as $j\rightarrow\infty$, then there exists a (relabelled) subsequence of $\{\nu_{j}\}_{j=1}^{\infty}$ such that $w_{j}:= \nabla H(\nu _{j})$ and $w:= \nabla H(\nu)$ admit the property that $w_{j} \rightarrow w$ $\mathscr{L}^{3}$-a.e. in $\mathbb{R}^3$ as $j \rightarrow \infty $.\\*[3mm]
{\bf (H3)} $H:\mathcal{P}_{ac}^2(\mathbb{R}^3) \rightarrow (-\infty , \infty ]$ is proper, lower semi-continuous and $\lambda$-convex on $\mathcal{P}^{2}_{ac}(\mathbb{R}^{3})$ for some $\lambda \in \mathbb{R}$.
\end{defn}
Condition {\bf (H1)} essentially requires that the growth of the velocity maps $\nabla H(\nu)$ is uniformly sublinear on bounded sets,
while condition {\bf (H2)} is a stability criterion. Condition {\bf (H3)} ensures that any dynamics $t\mapsto \nu_{t}$ which is `generated by' $H$ admits the property $H(\nu_{t})=H(\nu_{0})$ for all times t, which is typical of classical Hamiltonian systems on finite-dimensional symplectic manifolds. Indeed, as noted by Ambrosio and Gangbo, any Hamiltonian $H$ on $\mathcal{P}^{2}_{ac}(\mathbb{R}^{3})$ gives rise to the following abstract ODE thereon, 
\begin{equation}\label{abstractdyn}
\partial _t\nu_{t}+\nabla \cdot (J\nabla H(\nu_{t})\nu_{t})=0,
\end{equation}
where $J\in \mathbb{R}^{3\times 3}$ is the matrix above. The utility of this class of evolution equation in our context of the free-surface semi-geostrophic equations in geostrophic coordinates is that the vector field $J\nabla H(\nu_{t})$ is precisely the geostrophic velocity field at any time $t$. With this noted, let us now state in precise terms what we mean by a weak solution of the initial-value problem associated to \eqref{abstractdyn} above.
\begin{defn}[Global-in-time Weak Solution of \eqref{abstractdyn}]
Suppose an initial $\nu_{0}\in\mathcal{P}^{2}_{ac}(\mathbb{R}^{3})$ is given. We say that $\nu: [0, \infty)\rightarrow\mathcal{P}^{2}_{ac}(\mathbb{R}^{3})$ is an associated global-in-time weak solution of \eqref{abstractdyn} if and only if $t\mapsto\nu_{t}$ is absolutely continuous and satisfies
\begin{equation*}
\int_{0}^{\infty}\int_{\mathbb{R}^{3}}\left(\partial_{t}\phi+\nabla\phi\cdot J\nabla H(\nu_{t})\right)\,d\nu_{t}dt=0
\end{equation*}
for all $\phi\in C^{\infty}_{c}(\mathbb{R}^{3}\times (0, \infty))$. Moreover, $\lim_{t\rightarrow 0+}\nu_{t}=\nu_{0}$ in the narrow topology on $\mathcal{P}^{2}_{ac}(\mathbb{R}^{3})$.
\end{defn}
The strategy of the proof of Theorem \ref{new5.5} is to show that $H_{\mathcal{A}_{\ast}}$ is a Hamiltonian on $\mathcal{P}^{2}_{ac}(\mathbb{R}^{3})$ in the sense of definition \ref{hammy} above, and moreover that the map $J\nabla H_{\mathcal{A}_{\ast}}(\nu)$ coincides precisely with the geostrophic wind $U$ in geostrophic coordinates. It will then follow readily that the existence of a global-in-time weak solution of \eqref{abstractdyn} immediately implies the existence of a stable global-in-time weak solution of \eqref{fun}.

We begin by establishing the following basic properties of the map $H_{\mathcal{A}_{\ast}}$ on $\mathcal{P}^{2}_{ac}(\mathbb{R}^{3})$.
\begin{proposition}\label{newHok}
The map $\nu\mapsto -H_{\mathcal{A}_{\ast}}(\nu)$ is subdifferentiable, lower semi-continuous and $(-1)$-convex on $\mathcal{P}^{2}_{ac}(\mathbb{R}^{3})$.
\end{proposition}
\begin{proof}
Suppose $\mu\in\mathcal{P}^{2}_{ac}(\mathbb{R}^{3})$ is given and fixed. We shall show that $\partial H_{\mathcal{A}_{\ast}}(\mu)\neq \varnothing$. For ease of presentation, let us define the following maps:
\begin{itemize}
\item $\nabla P$ denotes the $e$-optimal map in $\mathcal{T}(\sigma_{h}, \nu)$, where $h=h(\nu)\mathcal{A}_{\ast}$ is the surface profile which minimises the free surface geostrophic energy;
\item $\nabla Q$ denotes the $e$-optimal map in $\mathcal{T}(\sigma_{k}, \mu)$, where $k=k(\mu)\in\mathcal{A}_{\ast}$ is the surface profile which minimises the free surface geostrophic energy;
\item $\nabla R$ denotes the $e$-optimal map in $\mathcal{T}(\sigma_{k}, \nu)$;
\item $\nabla S$ denotes the $c$-optimal map in $\mathcal{T}(\nu, \mu)$.
\end{itemize}
Suppose $\nu\in\mathcal{P}^{2}_{ac}(\mathbb{R}^{3})$ is arbitrary. One finds that
\begin{equation*}
\begin{array}{ll}
& -H_{\mathcal{A}_{\ast}}(\nu)+H_{\mathcal{A}_{\ast}}(\mu) \vspace{2mm} \\
=& \displaystyle-\int_{\Omega_{h}}e(\nabla P(x), x)\,dx+\int_{\Omega_{k}}e(\nabla Q(x), x)\,dx\vspace{2mm}\\
\geq & -\displaystyle \int_{\Omega_{k}}e(\nabla R(x), x)\,dx+\int_{\Omega_{k}}e(\nabla Q(x), x)\,dx\vspace{2mm}\\
= & \displaystyle-\int_{\mathbb{R}^{3}}e(y, \nabla R^{\ast}(y))\,d\nu(y)+\int_{\mathbb{R}^{3}}e(y, \nabla Q^{\ast}(y))\,d\mu(y) \vspace{2mm}\\
\geq & \displaystyle-\int_{\mathbb{R}^{3}}e(y, \nabla Q^{\ast}(\nabla S(y)))\,d\nu(y)+\int_{\mathbb{R}^{3}}e(y, Q^{\ast}(y))\,d\mu(y) \vspace{2mm} \\
=& \displaystyle -\int_{\mathbb{R}^{3}}e(\nabla S^{\ast}(y), \nabla Q^{\ast}(y))\,d\mu(y)+\int_{\mathbb{R}^{3}}e(y, Q^{\ast}(y))\,d\mu(y) \vspace{2mm} \\
\geq & \displaystyle\int_{\mathbb{R}^{3}}(\nabla Q(y)-(I-e_{3}\otimes e_{3})y)\cdot (\nabla S(y)-y)\,d\mu(y)-\frac{1}{2}W_{2}^{2}(\nu, \mu),
\end{array}
\end{equation*}
from which it follows by definition that $\nabla Q-(I-e_{3}\otimes e_{3})\mathrm{id}_{\mathbb{R}^{3}}\in \partial H_{\mathcal{A}_{\ast}}(\mu)$. As such, we deduce that $-H_{\mathcal{A}}$ is Fr\'{e}chet subdifferentiable on $\mathcal{P}^{2}_{ac}(\mathbb{R}^{3})$. Lower semi-continuity of $H_{\mathcal{A}_{\ast}}$ on $\mathcal{P}^{2}_{ac}(\mathbb{R}^{3})$ (with respect to the narrow topology thereon) follows from Cheng (\cite{cheng}, corollary 2.15). Finally, $(-1)$-convexity of $-H_{\mathcal{A}_{\ast}}$ follows from the $(-1)$-convexity of the map $\nu\mapsto -\frac{1}{2}W_{2}^{2}(\mu, \nu)$ for any fixed $\mu\in\mathcal{P}^{2}_{ac}(\mathbb{R}^{3})$ together with an application of (\cite{ambrosio2008gradient}, proposition 9.3.2).
\end{proof}
Let us now proceed to the proof our main result.
\begin{theorem}\label{newmainalt}
Let $1\leq p\leq\infty $ and $\nu _0\ll\mathscr{L}^{3}$ with density of class $L^p (\mathbb{R}^{3})$ and of compact support in $\mathbb{R} ^3$ be given. There exists a global-in-time $\mathcal{A}_{\ast}$-stable weak solution of \eqref{fun} associated to $\nu_{0}$.
\end{theorem}
\begin{proof}
The proof of this result comes in two parts. For the first part, we characterise the minimal element $\nabla H_{\mathcal{A}_{\ast}}(\nu)$ of the subdifferential $\partial H_{\mathcal{A}_{\ast}}(\nu)$ to ensure that the Hamiltonian $H_{\mathcal{A}_{\ast}}$ admits properties {\bf (H1)} and {\bf (H2)} and, in turn, that any weak solution $t\mapsto \nu_{t}$ of \eqref{abstractdyn} is indeed a $H_{\mathcal{A}_{\ast}}$-stable weak solution of \eqref{fun}. For the second part, we appeal to \cite{ambgangbo} to deduce the existence of a weak solution of the abstract evolution equation \eqref{abstractdyn}.

We follow the argument from (\cite{ambgangbo}, lemma 6.8). Suppose $\nu\in\mathcal{P}^{2}_{ac}(\mathbb{R}^{3})$ is given. To characterise the elements of $\partial H_{\mathcal{A}_{\ast}}(\nu)$, we let $\phi \in C_c^\infty (\mathbb{R}^3)$ and set
\begin{equation*}
g_{s}(y):=y+s\nabla\phi(y)
\end{equation*}
for $y\in\mathbb{R}^{3}$ and $s\in\mathbb{R}$. Note that for $|s|$ sufficiently small, $g_{s}$ is realised as the gradient of a convex function. We now define the measure $\nu _s := g_{s}\#\nu$, and denote by $h_{s}\in\mathcal{A}_{\ast}$ the map which minimises the argument in the free surface Hamiltonian expression $H_{\mathcal{A}_{\ast}}(\nu_{s})$, namely 
\begin{equation*}
H_{\mathcal{A}_{\ast}}(\nu_{s})=\inf_{T\in\mathcal{T}(\sigma_{h_{s}}, \nu_{s})}\int_{\Omega_{h_{s}}}e(x, T(x))\,dx.
\end{equation*}
Let $\xi \in \partial H(\nu )$.  %Since we seek velocity fields with vanishing third component, we can assume that $\xi _3 = 0$.  
 Combining the $(-1)-$concavity of $H_{\mathcal{A}_{\ast}}$ on $\mathcal{P}^{2}_{ac}(\mathbb{R}^{3})$ and making use of (\cite{ambgangbo}, proposition 4.2), we obtain
 \begin{equation}\label{newinequality}
-H_{\mathcal{A}_{\ast}}(\nu _s) + H_{\mathcal{A}_{\ast}}(\nu ) - \int_{\mathbb{R}^{3}}\xi(y)\cdot (R^{\nu _s}_{\nu}(y) - y)\,d\nu(y)+ \frac{1}{2}W_2^2(\nu , \nu _s) \geq 0,
  \end{equation}
where $R_{\nu}^{\nu_{s}}$ denotes the unique $e$-optimal map in $\mathcal{T}(\nu, \nu_{s})$. For $s\in\mathbb{R}$ with $|s|$ taken sufficiently small for the choice of $\phi\in C^{\infty}_{c}(\mathbb{R}^{3})$, we conclude that
 \begin{equation*}
 W_2^2(\nu , \nu _s) = \int_{\mathbb{R}^{3}} |y- R^{\nu _s}_{\nu }(y)|^2\, d\nu(y) = \int_{\mathbb{R}^{3}}|y - g_s(y)|^2\,d \nu (y) = s^2\int_{\mathbb{R}^{3}}|\nabla \phi (y)|^2 \,d\nu(y)
 \end{equation*}
 and
 \begin{equation*}
\int_{\mathbb{R}^{3}}\xi (y) \cdot (R^{\nu _s}_{\nu }(y)- y)\,d\nu (y) = \int_{\mathbb{R}^{3}} \xi(y) \cdot (g_s(y) - y)\,d\nu(y)= s\int_{\mathbb{R}^{3}} \xi (y) \cdot \nabla \phi(y)\,d\nu(y).
\end{equation*}
Combining this observation with (\ref{newinequality}), we therefore obtain
\begin{equation}\label{newfloaty}
\begin{array}{ll}
& \displaystyle -s\int_{\mathbb{R}^{3}}\xi(y)\cdot\nabla\phi(y)\,d\nu(y)+\frac{s^{2}}{2}\int_{\mathbb{R}^{3}}|\nabla \phi(y)|^{2}\,d\nu(y) \vspace{2mm} \\
\geq & \displaystyle -H_{\mathcal{A}_{\ast}}(\nu)+H_{\mathcal{A}_{\ast}}(\nu_{s}) \vspace{2mm} \\
\geq & \displaystyle -\int_{\mathbb{R}^{3}}e(g_{s}^{-1}(y), S_{\nu_{s}}^{\sigma_{h_{s}}}(y))\,d\nu_{s}(y)+\int_{\mathbb{R}^{3}}e(y, S_{\nu_{s}}^{\sigma_{h_{s}}}(y))\,d\nu_{s}(y),
\end{array}
\end{equation}
since $g_{s}\#\nu=\nu_{s}$. In the above, $S_{\nu_{s}}^{\sigma_{h_{s}}}$ denotes the unique $e$-optimal transport map from $\nu$ to $\sigma_{h_{s}}$. Noting that one has the expansion
\begin{equation*}
g_{s}^{-1}(y) = y - s\nabla \phi (y) + \frac{s^2}{2}\nabla ^2\phi (y)\nabla \phi (y) + \epsilon (s, y),
\end{equation*}
where $\epsilon$ is a function such that $|\epsilon (s, y)| \leqslant |s|^3\| \varphi \| _{C^3(\mathbb{R}^3)}$. Combining this expression for $g_s^{-1}$ with (\ref{newfloaty}), we conclude that
\begin{equation*}
\begin{array}{ll}
& \displaystyle -s\int_{\mathbb{R}^{3}}\xi(y)\cdot\nabla\phi(y)\,d\nu(y)+s^{2}\int_{\mathbb{R}^{3}}|\nabla\phi(y)|^{2}\,d\nu(y)\vspace{2mm}\\
\geq & \displaystyle s\int_{\mathbb{R}^{3}}(S_{\nu_{s}}^{\sigma_{h_{s}}}(y)-y)\cdot\nabla\phi(y)\,d\nu_{s}(y)+o(|s|)
\end{array}
\end{equation*}
as $|s|\rightarrow 0$. By definition of $g_s$ and $\nu _s$, one has that $\nu _s \rightarrow \nu$ with respect to the narrow topology on $\mathcal{P}_{ac}^{2}(\mathbb{R}^{3})$ as $s\to 0$.  Moreover, by the stability result in proposition \ref{chengres} above, we have that $\sigma_{h_{s}} \rightarrow \sigma_{h}$ in the narrow topology as $s \rightarrow 0$. Hence, dividing both sides first by $s>0$ (also $s<0$) and letting $|s| \rightarrow 0$, we use the stability of $e$-optimal transport maps to obtain
\begin{equation*}
-\int_{\mathbb{R}^{3}} \xi (y) \cdot \nabla \phi (y) \, d\nu(y)= \int_{\mathbb{R}^{3}} \left (S_{\nu}^{\sigma_{h}}(y) -y \right ) \cdot \nabla \phi (y)\,d\nu(y).
\end{equation*}
Thus, we have that $J(\pi_{\nu}(\xi)) = J\left(\mathrm{id}_{\mathbb{R}^{3}} - S_{\nu }^{\sigma_{h}}\right)$, where $\pi _\nu : L^2(\nu ; \mathbb{R}^{3}) \rightarrow T_\nu \mathcal{P}_{ac}^2(\mathbb{R}^{3})$ denotes the canonical orthogonal projection operator. We conclude that
\begin{equation}\label{newoursupdiff}
J(\nabla H(\nu)) =J(\mathrm{id}_{\mathbb{R}^{3}} - S_{\nu} ^{\sigma_{h}}).
\end{equation}
By using elementary properties of $e$-optimal transport maps, one can now check directly that conditions {\bf (H1)} and {\bf (H2)} on $H_{\mathcal{A}_{\ast}}$ hold true. Finally, by an application of (\cite{ambgangbo}, theorem 6.6), we conclude the existence of a global-in-time weak solution of \eqref{abstractdyn}. Owing to the characterisation result \eqref{newoursupdiff}, we may conclude that $t\mapsto \nu_{t}$ is in fact a global-in-time weak solution of \eqref{fun}.
\end{proof}

\section{Closing Remarks}
In this note, we chose the admissible class $\mathcal{A}$ to be $\mathcal{A}_{\ast}$, namely that which is generated by bounded continuous functions on $B$. It would be of interest to extend the main result of our article to a strictly larger class $\mathcal{A}\supset \mathcal{A}_{\ast}$ of free surface profiles (e.g. $\mathcal{A}$ generated by free surface profiles only in $L^{1}(B)$: see \cite{CKGParx}). However, as we have not shown that all minimisers of the functional \eqref{surfacefunk} (in most `reasonable' classes) should be continuous functions on $B$, it is therefore not obvious that the dynamics generated by $H_{\mathcal{A}}$ coincides with, in any sense, that generated by $H_{\mathcal{A}_{\ast}}$. Indeed, in many ways, the success of the theory we have proposed in this note is contingent upon $H_{\mathcal{A}}$ generating the same dynamics for all `reasonable' choices of $\mathcal{A}$. This remains an interesting open problem for future work.

Whilst the original initial-value problem of interest is that associated to \eqref{sgincom}, we have only been able to construct weak solutions of the initial-value problem associated to \eqref{fun}, which should be considered only as an auxiliary system. Ultimately, one would like to be able to construct solutions of \eqref{sgincom} by using solutions of \eqref{fun}. The only result to date which achieves this is for the special case that $\Omega_{h_{t}}$ is a convex subset of $\mathbb{R}^{3}$ (in fact, a fixed convex subset thereof), and is due to Ambrosio, Colombo, De Philippis and Figalli \cite{ambnew2}. As it is unclear (and most likely untrue) that $\Omega_{h_{t}}$ maintains its convexity at all times, if it is so endowed at time $t=0$, we cannot apply the techniques of \cite{ambnew2} to build a weak solution of the free-surface semi-geostrophic equations in Eulerian co-ordinates. Let us mention also that Caffarelli and McCann have developed in \cite{caffarelli2010free} a general theory of optimal transport in domains with free boundaries.  It would be interesting to investigate whether those results can be used to give an alternative proof of the problem considered here in this work.

It is physically correct, when modelling atmospheric flows as opposed to oceanic flows, to understand the analogue of system \eqref{sgincom} in which the Eulerian velocity field $u$ is {\em compressible}. We hope to consider this in future work.

\bibliographystyle{acm}
\bibliography{semigeostrophicbib}
\end{document}